\documentclass{gLMA2e}

\usepackage{epstopdf}% To incorporate .eps illustrations using PDFLaTeX, etc.
\usepackage{subfigure}% Support for small, `sub' figures and tables

\theoremstyle{plain} % Theorem-like structures
\newtheorem{theorem}{Theorem}[section]

\newtheorem{lemma}[theorem]{Lemma}
\newtheorem{proposition}[theorem]{Proposition}

\theoremstyle{definition}
\newtheorem{definition}{Definition}
\newtheorem{example}{Example}

\theoremstyle{remark}

\newcommand{\reals}{\mathbb{R}}
\newcommand{\ints}{\mathbb{I}\reals}

\usepackage{amsmath}
\usepackage{amssymb}
\usepackage{dsfont}
\usepackage{enumitem}
\usepackage{graphicx}
\usepackage{color}

\def\inum#1{\protect{\textrm{\boldmath $#1$}}} 

\def\Mid#1{#1^c}
\def\Rad#1{#1^\Delta}
\begin{document}\sloppy

%  Leave these commented lines here 
% \input{elaheader-volx-xx.tex}
% \setcounter{page}{1}

% \renewcommand{\thefootnote}{\fnsymbol{footnote}}
% \renewcommand{\thefootnote}{\arabic{footnote}}
% \renewcommand{\theequation}{\thesection.\arabic{equation}}

\title{Positive definiteness and stability of parametric interval matrices}
% Leave blank; editors will write the exact dates above

\author{
\name{Iwona Skalna$^{\ast}$\thanks{$^{\ast}$Corresponding author. Email: skalna@agh.edu.pl}}
\affil{AGH University of Science and Technology\\ ul. Gramatyka 10, 30-067 Krak{\'o}w, Poland \\
skalna@agh.edu.pl}
}
% Authors and running title to go on top of each page
\pagestyle{myheadings}
\markboth{I. Skalna}{Positive definiteness and stability of parametric interval matrices}

\maketitle
%
%%%%%%%%%%%
%
\begin{abstract}
We investigate positive definiteness, Hurwitz stability and Schur stability of parametric interval matrices. We give a~verifiable sufficient condition for positive definiteness of parametric interval matrices with non-linear dependencies. We also give several sufficient and necessary conditions for stability of symmetric parametric interval matrices with affine-linear dependencies. The presented results extend the results on positive definiteness and stability of interval matrices. In addition, we provide a~formula for the radius of stability of symmetric parametric interval matrices.
\end{abstract}

%\begin{classification}
%15B99,65G30,65G40.
%\end{classification}


\begin{classcode}
15B99; 65G40
\end{classcode}

%\begin{keywords}

%\end{keywords}

\bibliographystyle{plain}
\section{Introduction}
\label{chap:intro}
As it is well known, algebraic properties of real matrices, such as positive definiteness or stability can be verified in polynomial time using, e.g., Gaussian Elimination or Singular Value Decomposition. However, similar problems for interval matrices, and thus also for parametric interval matrices (note that an~$n$-dimensional parametric interval matrix can be considered as interval matrices with $n^2$ parameters) are, in general, NP-hard (see, e.g., \cite{Niemirowski:1993:SNP}, \cite{PoljakRohn:1993:CRNNP}, \cite{Rohn:1994:NPHR}). Therefore, in practical computation we must resort to sufficient and necessary condition that can be verified in a~reasonable amount of time.

In this paper we study positive definiteness, Hurwitz stability and Schur stability of parametric interval matrices. First, in Section \ref{sec:pos_def}, we give a~verifiable sufficient condition for positive definiteness of parametric interval matrices with both affine-linear and non-linear dependencies. We show the usefulness of the proposed approach for the interval global optimisation. 

Next, we present several sufficient and necessary conditions for Hurwitz stability (Section \ref{sec:stability}) and Schur stability (Section \ref{sec:shur_stability}) of symmetric parametric interval matrices. We present a~sufficient condition, which can be expected to work well in practical applications.

Finally, in Section \ref{sec:rad_stab}, we prove that the radius of stability of a~symmetric parametric interval matrix is equal to its radius of regularity \cite{Kolev:2014:RRRER}.
%
%%%%%%%%%%
%
\section{Notation and auxiliary results}
\label{sec:pim}
Throughout the paper, italic fonts are used to write real quantities, whereas bold italic fonts denote their interval counterparts. $\ints$ stands for the set of all real compact intervals, $\ints=\{\inum{x}=[\underline x,\overline x]\;|\;\underline x,\overline x\in\reals,\underline x\leqslant\overline x\}$. $\ints^n$ and $\ints^{n\times n}$ denote, respectively, the set of all interval vectors and the set of all interval matrices. The midpoint $\Mid{x}=(\underline x+\overline x)/2$ and the radius $\Rad{x}=(\overline x-\underline x)/2$ are applied to interval vectors and matrices componentwise. The minimal and maximal module of an interval $\inum{x}$ are defined, respectively, as $\langle\inum{x}\rangle=\min\{|x|\;|\;x\in\inum{x}\}$ and $|\inum{x}|=\max\{|x|\;|\;x\in\inum{x}\}$. The identity matrix of any size is denoted by~$I$, whereas $\rho(\cdot)$ and $\lambda$ stand, respectively, for the spectral radius and an eigenvalue of a~real matrix. 

\subsection{Parametric interval matrix}
\label{subs:pim}
A \emph{parametric interval matrix} $A(\inum{p})$, $\inum{p}\in\mathbb{IR}^K$, is defined as the following family of real parametric matrices
\begin{equation}
\label{eq:pim}
A(\inum{p})=\left\{A(p)\in\reals^{n\times n}\;|\;p\in\inum{p}\right\}.
\end{equation}
For each $p\in\inum{p}$, the elements of $A(p)$ are real valued functions of a~$K$-dimensional vector of parameters $p=(p_1,\ldots,p_K)\in\reals^K$, i.e. for $i,j=1,\ldots,n$,
\begin{equation}
\label{eq:pm_elem}
A_{ij}:\reals^K\ni (p_1,\ldots,p_K)\rightarrow A_{ij}(p_1,\ldots,p_K)\in\reals.
\end{equation}
Functions $A_{ij}$ ($i,j=1,\ldots,n$) can be generally divided into nonlinear and affine-linear. In the nonlinear case, it is usually assumed that $A_{ij}$ are continuous and differentiable on $\inum{p}$. In practical computation $A_{ij}$ are often given by closed-form expressions. 

In the affine-linear case, a~parametric interval matrix $A(\inum{p})$ can be represented as
\begin{equation}
\label{eq:pim_ldep}
A(\inum{p})=\left\{A^{(0)}+\textstyle\sum_{k=1}^KA^{(k)}p_k\;|\;p\in\inum{p}\right\},
\end{equation}
where $A^{(k)}\in\reals^{n\times n}$, $k=0,\ldots,K$.

The affine-linear case is by far not easy to handle, but to deal with nonlinear dependencies, some sophisticated methods for bounding the ranges of multivariate functions over a~box are required in addition (e.g., \cite{Neumaier:2015:IERF}). In \cite{Skalna:2017:SRPIM}, we have shown that with the use of affine arithmetic (AA) \cite{FigueiredoStolfi:1997:SVN}, the nonlinear case can be reduced to the affine-linear one. In what follows we will refer to this process as \textit{affine transformation}. Affine transformation of a~parametric interval matrix with affine-linear dependencies will be referred to as \textit{normalisation}. Obviously affine transformation cause some loss of information, but instead makes the computation far more simple.  
%
%----------------------------------------------------------------------------
\subsection{Regularity}
In general, a~parametric interval matrix $A(\inum{p})$ has a~certain property if all matrices $A(p)\in A(\inum{p})$ have that property. One of the most important properties of a~parametric interval matrix is its regularity.
\begin{definition}
\label{def:regpim}
An $n\times n$ parametric interval matrix $A(\inum{p})$ is regular if for each $p\in\inum{p}$ the matrix $A(p)$ is nonsingular.
\end{definition}

In the following theorem, we give a~verifiable sufficient condition for checking the regularity of parametric interval matrices. This theorem generalises the well-known result (due to Beeck \cite{Beeck:1975:ZPHI}) for interval matrices.
\begin{theorem}
\label{thm:reg_rho_cond}
Let $A(\inum{p})$ be a~square parametric interval matrix such that $A(\Mid{p})$ is non-singular and let $B=\left(\sum_{k=1}^K|A^{(k)}|p^{\Delta}_k\right)$. If
\begin{equation}
\label{eq:reg_rho_cond}
\rho\left(\left|A\left(\Mid{p}\right)^{-1}\right|B\right)<1,
\end{equation}
then $A(\inum{p})$ is regular.
\end{theorem}
\begin{proof}
Suppose to the contrary that $A(\inum{p})$ is singular. Then, there is a~nontrivial vector $x\in\mathbb{R}^n$ such that $A(p)x=0$. Hence
$$
\left|A(\Mid{p})x\right|=\left|(A(\Mid{p})-A(p))x\right|
=\left|\left\{\sum_{i=1}^KA^{(k)}(\Mid{p}_k-p_k)\right\}\right|\leqslant B|x|.
$$
Without loss of generality we can assume that $x=A(\Mid{p})^{-1}y$, for some $y\neq 0$. Hence,
$$
|y|\leqslant B|A(\Mid{p})^{-1}||y|,
$$
and thus (Corollary 3.2.3 and Proposition 3.2.4, \cite{Neumaier:1990:ILS}) 
$$
1\leqslant \rho\left(B\left|A(\Mid{p})^{-1}\right|\right)=
\rho\left(\left|A(\Mid{p})^{-1}\right|B\right),
$$ 
which yields a~contradiction.
\end{proof}
The above sufficient condition can be verified in polynomial time even with a~large number of parameters. The spectral decomposition can be done in $\mathcal{O}(n^3)$ (for symmetric matrices even faster using the Wilkinson's algorithm, which has $\mathcal{O}(n^2)$ time complexity), matrix inverse can also be done in $\mathcal{O}(n^3)$ (or even faster using, e.g. Coppersmith-Winograd algorithm has $\mathcal{O}(n^{2.376})$ time complexity). Hence, we obtain the $\mathcal{O}(n^3)$ time complexity. In order to avoid rounding errors in numerical computation, instead of the condition (\ref{eq:reg_rho_cond}) we recommend to use  the condition given in \cite{Skalna:2017:SRPIM} (formula (9)) with $R\approx A(\Mid{p})^{-1}$.

The regularity property gives us some quantitative knowledge about the matrix. Whereas sometimes, as pointed by Kolev \cite{Kolev:2014:RRRER}, it might be useful to have a~quantity, which would measure the distance of a~parametric matrix from singularity. The \textit{radius of regularity} of a~normalised parametric interval matrix with affine-linear dependencies, defined as (\cite{Kolev:2014:RRRER})
\begin{equation}
\label{eq:reg_rad_kolev}
r(A(\inum{p}))=\min\{r\geqslant 0\mid
\exists p\in\inum{p}\;A(rp)\textrm{ is singular}\},
\end{equation}
can be considered as such measure. Kolev \cite{Kolev:2014:RRRER} proposed an explicit formula for $r(A(\inum{p}))$, which can be useful to develop an algorithm for computing the radius of regularity, which is not \textit{a priori} exponential.

%
%%%%%%%%%%
%
%========================================================================
%
\section{Positive definiteness}
%
%========================================================================
\label{sec:pos_def}
%------------------------------------------------------------------------
%
Positive definiteness of parametric interval matrices is important, e.g., for speeding up interval global optimisation. As it is well known, the convexity of a~function $f$ on a~box $\inum{x}$ can be verified by checking the positive definiteness of all Hessian matrices $\nabla^2f(x)$, where $x\in\inum{x}$. Hence, if $\nabla^2\inum{f}(\inum{x})$ is positive definite, then $f$ is strictly convex on $\inum{x}$.
\begin{definition}
A~parametric interval matrix $A(\inum{p})$ is strongly positive definite (semi-definite) if $A(p)$ is positive definite (semi-definite) for each $p\in\inum{p}$.
\end{definition}

\begin{definition}
A~parametric interval matrix $A(\inum{p})$ is weakly positive definite (semi-definite) if $A(p)$ is positive definite (semi-definite) for at least one  $p\in\inum{p}$.
\end{definition}

The problem of checking whether a~parametric interval matrix is strongly positive definite is co-NP hard (cf. \cite{KreinEtAl:1998:CCF}). Some sufficient and necessary conditions for the positive definiteness of parametric interval matrices with affine-linear dependencies were given by Hlad{\'i}k \cite{Hladik:2017:PSPDLPM}. They are especially useful when the number of parameters is small enough. For symmetric matrices, he provided a~verifiable sufficient condition, which relies on checking the positive definiteness of some specific real matrix. Below, we present another verifiable sufficient condition for checking positive definiteness, which does not assume the symmetry of the parametric interval matrix. The presented result generalises the result for interval matrices (cf. \cite{Rohn:1994:PDS}).
\begin{theorem}
\label{thm:strposdef_suffcond}
Let $A(\inum{p})$ be a~parametric interval matrix such that $A(\Mid{p})$ is positive definite and (\ref{eq:reg_rho_cond}) holds true. Then $A(\inum{p})$ is strongly positive definite.
\end{theorem}
\begin{proof}
The condition (\ref{eq:reg_rho_cond}) guarantees the regularity of $A(\inum{p})$. Since $A(\Mid{p})$ is positive definite and $A(\inum{p})$ is regular, hence (Theorem 10, \cite{Hladik:2017:PSPDLPM}) $A(\inum{p})$ is strongly positive definite.
\end{proof} 
\begin{example}
Hlad{\'i}k \cite{Hladik:2017:PSPDLPM} considered cubic forms, since the entries of the Hessian of a~cubic form are affine-linear functions of variables. However, using the affine transformation, we can extend the applicability of the above theorem to any computable function. Let us consider the following function of $x=(x_1,x_2,x_3)$ (cf.~\cite{Hladik:2017:PSPDLPM}):
$$
f(x)=x_1^4+2x_1^2x_2-x_1x_2x_3+3x_2x_3^2+5x_2^3
$$
with $x\in\inum{x}=([2,3],[1,2],[0,1])$. The Hessian matrix
$$
\nabla^2f(x_1,x_2,x_3)=
\left(
\begin{array}{ccc}
12x_1^2+4x_2\; & 4x_1-x_3\; & -x_2 \\
4x_1-x_3\; & 30x_2\; & 6x_3-x_1 \\
-x_2\; & 6x_3-x_1\; & 6x_2
\end{array}
\right).
$$
The affine transformation of $\nabla^2f(x_1,x_2,x_3)$ yields a~parametric matrix with affine-linear dependencies
$$
A(e)=
\left(
\begin{array}{ccc}
82.5+30\varepsilon_1+2\varepsilon_2+1.5\varepsilon_4\;\; & 9.5+2\varepsilon_1-0.5\varepsilon_3\; & -1.5-0.5\varepsilon_2 \\
9.5+2\varepsilon_1-0.5\varepsilon_3\; & 45+15\varepsilon_2\; & 0.5-0.5\varepsilon_1+3\varepsilon_3 \\
-1.5-0.5\varepsilon_2\; & 0.5-0.5\varepsilon_1+3\varepsilon_3\; & 9+3\varepsilon_2
\end{array}
\right).
$$
Since $A(\Mid{p})=A^{(0)}$ is positive definite with eigenvalues (rounded to two decimal places) $(8.96,42.75,84.79)$ and $\rho(|(A^{(0)})^{-1}|B)\approx 0.61<1$, hence, by Theorem \ref{thm:strposdef_suffcond}, $\nabla^2\inum{f}(\inum{x})$ is positive definite, and thus $f$ is strictly convex on~$\inum{x}$.
\end{example}

%========================================================================
%
\section{Stability}
%
%========================================================================
\label{sec:stability}
%------------------------------------------------------------------------
A~real square matrix $A$ is called \textit{stable matrix} (also \textit{Hurwitz matrix}) if all its eigenvalues has strictly negative real part, i.e., if for each eigenvalue $\lambda_i$ of $A$, $\mathrm{Re}(\lambda_i)<0$. The problem of stability of interval matrices is strictly connected with the behaviour of a~linear time invariant system $\dot{x}(t)=Ax(t)$ under perturbation (see, e.g., \cite{Mansour:1989:RSIM} and the references therein). As in \cite{Rohn:1994:PDS}, we investigate here mainly stability of symmetric parametric interval matrices. The great advantage of the parametric approach is that a~parametric interval matrix contains only symmetric matrices, whereas an~interval matrix can contain also non-symmetric matrices, whose eigenvalues might not be real, and therefore requires some additional care~\cite{Rohn:1994:PDS}.

\begin{definition}
A~parametric interval matrix $A(\inum{p})$ is stable if, for each $p\in\inum{p}$, $A(p)$ is stable.
\end{definition}

\begin{lemma}[cf. \cite{Fiedler:1986:SMTANM}]
\label{lem:posdef_stab_equiv}
A~real symmetric matrix $A$ is stable if and only if $-A$ is positive definite.
\end{lemma}

In the first theorem we give some sufficient and necessary conditions for the stability of parametric interval matrices (cf. \cite{Rohn:1994:PDS}).
\begin{theorem}
\label{thm:stab_posdef}
Let $A(\inum{p})$ be a~symmetric parametric interval matrix. Then, the following assertions are equivalent:
\begin{itemize}
\item[(i)] $A(\inum{p})$ is stable,
\item[(ii)] $A(p)$ is stable for each $p$ such that $p_k\in
\{\underline p_k,\overline p_k\}$ (vertex property),
\item[(iii)] $-A(\inum{p})$ is strongly positive definite,
\end{itemize}
\end{theorem}
\begin{proof}
By Lemma \ref{lem:posdef_stab_equiv} (i)$\Leftrightarrow$(iii). So, it is enough to prove (i)$\Rightarrow$(ii)$\Rightarrow$(iii).

(i)$\Rightarrow$(ii) The proof is obvious, since for $A(p)$ such that $p_k\in\{\underline p_k,\overline p_k\}$, it holds that $A(p)\in A(\inum{p})$.

(ii)$\Rightarrow$(iii) Let $p$ be such that $p_k\in\{\underline p_k,\overline p_k\}$. Then $A(p)$ is stable, and thus, by Lemma \ref{lem:posdef_stab_equiv}, $-A(p)$ is positive definite. Since $p$ was chosen arbitrarily, therefore $-A(\inum{p})$ is strongly positive definite (cf. Theorem 7, \cite{Hladik:2017:PSPDLPM}).
\end{proof}

%====================
In the next theorem we formulate another sufficient and necessary condition, which gives a~link between stability and regularity of symmetric parametric interval matrix (cf. \cite{Rohn:1994:PDS}).
\begin{theorem}
\label{thm:stab_reg_relation}
A~symmetric parametric interval matrix $A(\inum{p})$ is stable if and only if the following assertions hold:
\begin{itemize}
\item[(i)] $A(p)$ is stable for an arbitrary $p\in\inum{p}$,
\item[(ii)] $A(\inum{p})$ is regular.
\end{itemize}
\end{theorem}
\begin{proof}
The ``only if'' part is obvious, since each stable matrix is non-singular. Conversely, if $A(p)$ is stable, then $-A(p)$ is positive definite, and since $A(inum{p})$ is regular, hence $-A(\inum{p})$ is regular as well. Thus, (Theorem 7, \cite{Hladik:2017:PSPDLPM}) $-A(\inum{p})$ is strongly positive definite, and hence $A(\inum{p})$ is stable by Theorem~\ref{thm:stab_posdef}.
\end{proof}
%====================

\begin{theorem}
\label{thm:stab_suffcond}
Let $A(\inum{p})$ be a~symmetric parametric interval matrix such that $A(\Mid{p})$ is stable and (\ref{eq:reg_rho_cond}) holds true. Then $A(\inum{p})$ is stable.
\end{theorem}
\begin{proof}
The theorem follows directly from Theorem \ref{thm:strposdef_suffcond} applied to $-A(\inum{p})$.
\end{proof}

%========================================================================
%
\section{Schur stability}
%
%========================================================================
\label{sec:shur_stability}
A~real square matrix $A$ is called \textit{Schur stable} if all its eigenvalues lie in a~unit circle, i.e., if $|\lambda_i|<1$, for each eigenvalue $\lambda_i$ of $A$. Schur matrices are strongly connected with asymptotic stability of polynomials and dynamical systems. Similar as in the previous section, we consider here only symmetric parametric interval matrices (cf. \cite{Rohn:1994:PDS}).

\begin{definition}
A~parametric interval matrix $A(\inum{p})$ is Schur stable if, for each $p\in\inum{p}$, $A(p)$ is Schur stable.
\end{definition}

The first theorem we present gives necessary and sufficient condition for a~symmetric parametric interval matrix to be Schur stable. Similar result for interval matrices was given in \cite{Hertz:1992:EESRSIIM}.
\begin{theorem}
\label{thm:schur_stab1}
Let $A(\inum{p})$ be a~symmetric parametric interval matrix. Then, the following assertions are equivalent:
\begin{itemize}
\item[(i)] $A(\inum{p})$ is Schur stable,
\item[(ii)] $A(p)$ is Schur stable for each $p$ such that $p_k\in\{\underline p_k,\overline p_k\}$.
\end{itemize}
\end{theorem}
\begin{proof}
The implication (i)$\Rightarrow$(ii) is obvious. To prove (ii)$\Rightarrow$(i), take arbitrary $A(p')\in A(\inum{p})$, $x\neq 0$, and put $s_k=\mathrm{sgn}(x^TA^{(k)}x)$. Since 
$$
\left|x^T(A(p')-A(\Mid{p}))x\right|\leqslant
\sum_{k=1}^K\left|x^TA^{(k)}x\right|\Rad{p}_k,
$$ 
hence 
\begin{flalign*}
x^TA(p')x &=x^TA(\Mid{p})x+x^T(A(p')-A(\Mid{p}))x\leqslant \\
& \leqslant x^TA(\Mid{p})x+\sum_{k=1}^K\left|x^TA^{(k)}x\right|\Rad{p}_k\leqslant \\
& \leqslant x^T\left(A(\Mid{p})+\sum_{k=1}^KA^{(k)}s_k\Rad{p}_k\right)x,
\end{flalign*}
and we have
$$
\frac{x^TA(p')x}{x^Tx}\leqslant \frac{x^T\left(A(\Mid{p})+\sum_{k=1}^KA^{(k)}s_k\Rad{p}_k\right)x}{x^Tx}=
\frac{x^TA(p)x}{x^Tx},
$$
where $p$ is such that $p_k\in\{\underline p_k,\overline p_k\}$, which implies that
$$
\lambda_{\max}(A(p'))\leqslant\lambda_{\max}(A(p))<1.
$$
Analogically, we can prove that $\lambda_{\min}(A(p'))>-1$, and hence $A(p')$ is Schur stable. Since $A(p')$ was chosen arbitrarily, thus $A(\inum{p})$ is Schur stable.
\end{proof}

The next theorem gives another sufficient and necessary condition for Schur stability of symmetric parametric interval matrices. It also links Schur stability and Hurwiz stability (cf. \cite{Rohn:1994:PDS}).

\begin{theorem}
\label{thm:schurstab_hurwitzstab}
Let $A(\inum{p})$ be a~symmetric parametric interval matrix. Then $A(\inum{p})$ is Schur stable if and only if the parametric interval matrices
$$
A(\inum{p})-I\quad\textrm{and}\quad-A(\inum{p})-I
$$
are stable.
\end{theorem}
\begin{proof}
To prove the ``only if'' part, take arbitrary $p\in\inum{p}$. Since $A(p)$ is Schur stable, hence it has all its eigenvalues in $(-1,1)$. Therefore, all eigenvalues of $A(p)-I$ lie in $(-2,0)$, and hence $A(\inum{p})-I$ is stable. To prove the stability of the matrix $-A(\inum{p})-I$, it is enough to observe that if $A(p)$ is Schur stable, then $-A(p)$ is Schur stable as well.

To prove the ``if'' part, take $A(p)\in A(\inum{p})$ and let $\lambda$ be its eigenvalue. Since $A(p)-I$ is stable, hence all its eigenvalues are negative, hence $\lambda-1<0$, which means that $\lambda<1$. Moreover, $-\lambda$ is an~eigenvalue of $-A(p)$, and since $-A(p)-I$ is stable, hence, using the same reasoning, we obtain that $\lambda>-1$. Since $A(p)$ was chosen arbitrarily, hence $A(\inum{p})$ is Schur stable.
\end{proof}

The next theorem gives a~sufficient condition for Schur stability of symmetric parametric interval matrices. The validity of the theorem follows directly from Theorem \ref{thm:stab_suffcond} and Theorem \ref{thm:schurstab_hurwitzstab}.
\begin{theorem}
\label{thm:schurstab_suffcond}
Let $A(\inum{p})$ be a~symmetric parametric interval matrix such that $A(\Mid{p})$ is Schur stable and the (\ref{eq:reg_rho_cond}) holds true. Then $A(\inum{p})$ is Schur stable.
\end{theorem}

%========================================================================
%
\section{Radius of stability}
%
%========================================================================
\label{sec:rad_stab}
By analogy to the radius of stability of an~interval matrix (), we define the radius stability of a~parametric interval matrix as 
\begin{equation}
\label{eq:rad_of_stability}
s(A(\inum{p}))=\min\{r\geqslant 0\mid
\exists p\in\Mid{p}+r[-\Rad{p},\Rad{p}]\;A(p)\textrm{ is unstable}\}.
\end{equation}
Obviously, if $s(A(\inum{p}))=0$, then $A(\inum{p})$ is unstable, and if $s(A(\inum{p}))>1$, then $A(\inum{p})$ is stable.

\begin{proposition}
Let $A(\inum{p})$ be a~symmetric parametric interval matrix such that $A(\Mid{p})$ is stable. Then
\begin{equation}
s(A(\inum{p}))=r(A(\inum{p})).
\end{equation}
\end{proposition}
\begin{proof}
The proposition follows directly from Theorem \ref{thm:stab_reg_relation} and Theorem 1 in~\cite{Kolev:2014:RRRER}.
\end{proof}
%
%%%%%%%%
%

%\section*{Acknowledgements} The author would like to thank the anonymous reviewers for their helpful and constructive comments that greatly contributed to improving the overall quality of this paper. The author also likes to thank the editors for their generous comments and support during the review process.

\end{document}